\documentclass[a4paper]{scrartcl}
\usepackage[T1]{fontenc}
\usepackage{amsmath}
\usepackage{amssymb,bbm}
\usepackage{amsthm}
\usepackage{enumerate}

\newcommand{\Z}{\mathbb{Z}}
\newcommand{\R}{\mathbb{R}}
\newcommand{\C}{\mathbb{C}}
\newcommand{\Skew}{A \rtimes_\sigma G}

\DeclareMathOperator{\identity}{id}
\DeclareMathOperator{\Aut}{Aut}
\DeclareMathOperator{\Supp}{Supp}
\DeclareMathOperator{\Ideal}{Ideal}
\DeclareMathOperator{\cl}{cl}
\DeclareMathOperator{\Homeo}{Homeo}
\DeclareMathOperator{\Stab}{Stab}

\theoremstyle{plain}
\newtheorem{theorem}{Theorem}[section]
\newtheorem{lemma}[theorem]{Lemma}
\newtheorem{proposition}[theorem]{Proposition}
\newtheorem{corollary}[theorem]{Corollary}
\theoremstyle{definition}
\newtheorem{definition}{Definition}[section]
\newtheorem{example}{Example}[section]
\newtheorem{remark}{Remark}

\begin{document}

\author{Johan \"{O}inert\footnote{Address: Department of Mathematical Sciences, University of Copenhagen, Universitetsparken 5, DK-2100 Copenhagen \O, Denmark, E-mail: oinert@math.ku.dk}}

\title{Simplicity of skew group rings of abelian groups}

\maketitle

\begin{abstract}
Given a group $G$, a (unital) ring $A$ and a group homomorphism $\sigma : G \to \Aut(A)$,
one can construct the skew group ring $\Skew$.
We show that a skew group ring $\Skew$, of an abelian group $G$,
is simple if and only if its centre is a field and $A$ is $G$-simple.
If $G$ is abelian and $A$ is commutative,
then $\Skew$ is shown to be simple if and only if $\sigma$ is injective
and $A$ is $G$-simple.
As an application
we show that a transformation group
$(X,G)$, where $X$ is a compact Hausdorff space and $G$ is abelian,
is minimal and faithful if and only if its associated
skew group algebra $C(X)\rtimes_\sigma G$ is simple.
We also provide an example of a skew group algebra, of an (non-abelian) ICC group,
for which the above conclusions fail to hold.
\end{abstract}

\pagestyle{headings}


\section{Introduction}

Given a group $G$, a (unital) ring $A$ and a group homomorphism $\sigma : G \to \Aut(A)$,
one can construct the skew group ring $\Skew$ (see Section \ref{Section:Preliminaries} for details).
Skew group rings serve as an elementary
way of constructing non-commutative rings.
They occur naturally in different branches of mathematics,
e.g.
in the representation theory of Artin algebras \cite{ReitenRiedtmann1985},
in the computation of Grothendieck groups \cite{AuslanderReiten1986},
in the study of singularities \cite{Auslander1986,AuslanderReiten1987},
in orbifold theory \cite{Yamskulna2002} and
in the Galois theory of skew fields \cite{XuShum2008}.
Recently, skew group rings have proven to be an important tool
in the investigation of Calabi-Yau algebras derived from superpotentials \cite{WuZhu2011}
and in the representation theory of certain preprojective algebras \cite{Demonet10,HouYang11}.

The ideal structure of skew group rings has been studied in depth
(see e.g. \cite{CrowThesis,Crow05,FisherMontgomery78,Mihovski2001,MontgomeryBook,OsterburgPassman1992,OsterburgPeligrad1991,PassmanBook,ZaleskiiNeroslavskii1975}).
Nevertheless, necessary and sufficient conditions for simplicity
of a general skew group ring are not known.

The present author has shown, in his thesis \cite[Theorem E.3.5]{OinertThesis} (see also \cite{Oinert}),
that for a skew group ring $\Skew$ over
a commutative ring $A$, the subring $A$ is a maximal
commutative subring of $\Skew$ if and only if $A$ has the so called \emph{ideal intersection property}
in $\Skew$, i.e. each non-zero ideal
of $\Skew$ intersects $A$ non-trivially. From this one obtains the
following characterization of simplicity of skew
group rings over commutative rings (\cite[Theorem E.6.13]{OinertThesis}).

\begin{theorem}[\cite{Oinert,OinertThesis}]\label{SimplicityMaxComm}
Let $\Skew$ be a skew group ring where $A$ is a commutative ring.
The following two assertions are equivalent:
\begin{enumerate}[{\rm (i)}]
	\item $\Skew$ is a simple ring;
	\item $A$ is $G$-simple and $A$ is a maximal commutative subring of $\Skew$.
\end{enumerate}
\end{theorem}

In this article we instead turn the focus to the case when $A$ is arbitrary, but
$G$ is abelian. Under the assumption that $G$ is abelian and $A$
is $G$-simple, we show that every non-zero ideal of $\Skew$ contains
a non-zero central element (Proposition \ref{CenterIntersection}).
Using this we are able to give a characterization of simplicity
of skew group rings of abelian groups, and this is the first main result of this article.

\begin{theorem}\label{MainTheorem}
Let $\Skew$ be a skew group ring.
Consider the following three assertions:
\begin{enumerate}[{\rm (i)}]
	\item\label{MainTheorem:Simple} $\Skew$ is a simple ring;
	\item\label{MainTheorem:FieldGSimple} $A$ is $G$-simple and $Z(\Skew)$ is a field;
	\item\label{MainTheorem:InjectiveGSimple} $A$ is $G$-simple and $\sigma : G \to \Aut(A)$ is injective.
\end{enumerate}
The following conclusions hold:
\begin{enumerate}[{\rm (a)}]
	\item\label{MainTheorem:A} \eqref{MainTheorem:Simple} implies \eqref{MainTheorem:FieldGSimple} and \eqref{MainTheorem:InjectiveGSimple};
	\item\label{MainTheorem:B} in general \eqref{MainTheorem:FieldGSimple} and \eqref{MainTheorem:InjectiveGSimple} do not imply \eqref{MainTheorem:Simple};
	\item\label{MainTheorem:C} if $G$ is an abelian group, then \eqref{MainTheorem:Simple} is equivalent to \eqref{MainTheorem:FieldGSimple};
	\item\label{MainTheorem:D} if $G$ is an abelian group and $A$ is a commutative ring, then \eqref{MainTheorem:Simple}, \eqref{MainTheorem:FieldGSimple}
	and \eqref{MainTheorem:InjectiveGSimple} are all equivalent.
\end{enumerate}
\end{theorem}

In 1978 Power showed \cite{Power} that a topological dynamical system $(X,\Z)$ (of an infinite compact Hausdorff space $X$)
is minimal if and only if its associated crossed product $C^*$-algebra $C^*(C(X)\rtimes \Z)$ is simple.
The present author showed in \cite[Theorem E.7.6]{OinertThesis} (see also \cite{Oinert}) that, analogously,
minimality of $(X,\Z)$ is equivalent to simplicity
of the skew group algebra $C(X)\rtimes \Z$.
Recently it was shown by de Jeu, Svensson and Tomiyama \cite{ChristianBanach}
that the analogous result also holds for the crossed product Banach algebra $\ell^1( C(X) \rtimes \Z)$.

Let $X$ be a compact Hausdorff space and $G\curvearrowright X$ a strongly continuous action,
inducing a group homomorphism $\sigma : G \to \Aut(C(X))$ (see Section \ref{DynamicalSection} for details).
This allows us to define the skew group algebra $C(X) \rtimes_\sigma G$, and
as an application of Theorem \ref{MainTheorem} we obtain the second main result of this article
which is a generalization
of the aformentioned (algebraic) result on $(X,\Z)$.

\begin{theorem}\label{TransGroupTheorem}
Let $(X,G)$ be a transformation group of a compact Hausdorff space $X$.
Consider the following five assertions:
\begin{enumerate}[{\rm (i)}]
	\item\label{TransGroupTheorem:Simple} $C(X) \rtimes_\sigma G$ is a simple algebra;
	\item\label{TransGroupTheorem:GSimpleMaxComm} $C(X)$ is $G$-simple and $C(X)$ is a maximal commutative subalgebra of $C(X) \rtimes_\sigma G$;
	\item\label{TransGroupTheorem:GSimpleField} $C(X)$ is $G$-simple and $Z(C(X) \rtimes_\sigma G)$ is a field; 
	\item\label{TransGroupTheorem:GSimpleInjective} $C(X)$ is $G$-simple and $\sigma : G \to \Aut(C(X))$ is injective;
	\item\label{TransGroupTheorem:MinimalFaithful} $(X,G)$ is minimal and faithful.
\end{enumerate}
The following conclusions hold:
\begin{enumerate}[{\rm (a)}]
  \item\label{TransGroupTheorem:A} \eqref{TransGroupTheorem:Simple} and \eqref{TransGroupTheorem:GSimpleMaxComm} are equivalent and imply \eqref{TransGroupTheorem:GSimpleField}, \eqref{TransGroupTheorem:GSimpleInjective} and \eqref{TransGroupTheorem:MinimalFaithful};
  \item\label{TransGroupTheorem:B} \eqref{TransGroupTheorem:GSimpleInjective} and \eqref{TransGroupTheorem:MinimalFaithful} are equivalent;
  \item\label{TransGroupTheorem:C} in general \eqref{TransGroupTheorem:GSimpleField}, \eqref{TransGroupTheorem:GSimpleInjective}
  and \eqref{TransGroupTheorem:MinimalFaithful} do not imply \eqref{TransGroupTheorem:Simple} and \eqref{TransGroupTheorem:GSimpleMaxComm};
  \item\label{TransGroupTheorem:D} if $G$ is an abelian group, then
  \eqref{TransGroupTheorem:Simple}, \eqref{TransGroupTheorem:GSimpleMaxComm},
  \eqref{TransGroupTheorem:GSimpleField}, \eqref{TransGroupTheorem:GSimpleInjective}
  and \eqref{TransGroupTheorem:MinimalFaithful} are all equivalent.
\end{enumerate}
\end{theorem}

It is natural to ask whether this connection between minimality, faithfulness, freeness and simplicity
translates to crossed product $C^*$-algebras.
If $(X,G)$ is a second countable locally compact transformation group
with $G$ amenable and freely acting, then it is known (see \cite{WilliamsBook}) that
the crossed product $C^*$-algebra $C_0(X) \rtimes G$
is simple if and only if $G$ acts minimally on $X$.
If a group $G$ acts on a (Borel) measurable space $X$, in such a way
that the action is non-singular, free and ergodic, then Murray and von Neumann
have shown (see e.g. \cite{vonNeumann}) that the crossed product von Neumann algebra $L^\infty(X)\rtimes G$ is
a factor, i.e. simple. 

\section{Preliminaries}\label{Section:Preliminaries}

Let $A$ be a
unital and
associative ring, $G$ a multiplicatively written group with neutral element $e$
and $\sigma : G \to \Aut(A)$ a group homomorphism.
The triple $(A,G,\sigma)$ gives rise to a \emph{skew group ring}, denoted $\Skew$, in the following way.
Let $\{u_g\}_{g\in G}$ be a copy of $G$ (as a set) and define $\Skew$ as the free left $A$-module with basis $\{u_g\}_{g\in G}$.
Addition is defined by
$\sum_{g\in G} a_g u_g + \sum_{h\in G} b_h u_h := \sum_{g\in G} (a_g+b_g) u_g$
for $\sum_{g\in G} a_g u_g, \sum_{h\in G} b_h u_h \in \Skew$. Multiplication is defined as the bilinear extension of the rule
\begin{equation}\label{ProductRule}
	(a_g u_g)(b_h u_h) = a_g \sigma_g(b_h) u_{gh}
\end{equation}
for $g,h\in G$ and $a_g,b_h\in A$.
These two operations make $\Skew$ into a unital and associative ring.
The multiplicative identity in $\Skew$ is given by $1_A u_e$, but by abuse of notation we shall simply write $1$.
It follows from \eqref{ProductRule} that $u_g u_{g^{-1}} = u_{g^{-1}} u_g = 1_A u_e$ and hence $u_g^{-1} = u_{g^{-1}}$, for $g\in G$.
By putting $R_g := A u_g$, for $g\in G$, we see that $\Skew = \oplus_{g\in G} R_g$ is a strongly $G$-graded ring.
Each element $r$ of $\Skew$ may be written uniquely as $r=\sum_{g\in G} a_g u_g$ for some $a_g\in A$, for $g\in G$,
of which all but finitely many are zero. The support of $r$, denoted $\Supp(r)$, is defined as the finite set $\{g\in G \mid a_g \neq 0\}$
and its cardinality will be denoted by $|\Supp(r)|$.
The multiplication rule \eqref{ProductRule} yields $u_g a = \sigma_g(a) u_g$ for all $g\in G, a\in A$.
This means that, for each $g\in G$, the map $\sigma_g$ is implemented by the basis elements of $\Skew$, i.e.
\begin{displaymath}
	\sigma_g(a)=u_g a u_g^{-1}, \quad \forall a\in A.
\end{displaymath}
An ideal $I$ of $A$ is said to be \emph{$G$-invariant} if $\sigma_g(I)\subseteq I$ holds for all $g\in G$.
If $A$ and $\{0\}$ are the only $G$-invariant ideals of $A$, then $A$ is said to be \emph{$G$-simple}.
The \emph{fixed ring} of $A$ is defined as the set $A^G := \{a\in A \mid \sigma_g(a)=a, \,\, \forall g\in G\}$.

Given a subgroup $H$ of $G$ we denote by $A \rtimes_\sigma H$ the subring of $\Skew$
consisting of all elements $r\in \Skew$ which satisfy $\Supp(r)\subseteq H$.
The centralizer of a subset $S$ of a ring $T$ will be denoted by $C_T(S)$
and is defined as the set of all elements of $T$ that commute with each element of $S$.
If $S$ is a commutative subring of $T$ and $C_T(S)=S$ holds,
then $S$ is said to be a \emph{maximal commutative subring} of $T$.
The centre of $T$ will be denoted by $Z(T)$.

\noindent An automorphism $\varphi$ of a ring $T$ is said to be \emph{inner} if there exists a unit $v\in T$ such that $\varphi(t)=v t v^{-1}$ holds for all $t\in T$, and \emph{outer} otherwise.
The group homomorphism $\sigma : G \to \Aut(A)$, or simply $G$, is said to be \emph{inner} if $\sigma_g \in \Aut(A)$ is inner for some $g\in G\setminus \{e\}$,
and \emph{outer} otherwise.

We shall make use of the following two maps of abelian groups:\\
$\epsilon : \Skew \to A$, $\sum_{g\in G} a_g u_g \mapsto \sum_{g\in G} a_g$;
and $E: \Skew \to A$, $\sum_{g\in G} a_g u_g \mapsto a_e$.\\
The map $\epsilon$ is known as the \emph{augmentation map}
and if $\ker(\sigma)=G$, then $\epsilon$ is actually a ring morphism.

\section{Necessary conditions for simplicity of $\Skew$}

The following proposition gives some, presumably well-known, necessary conditions for simplicity of a general skew group ring.
For the sake of completeness, we include the proof.

\begin{proposition}\label{NecessaryConditions}
Let $R=\Skew$ be a skew group ring.
If $\Skew$ is simple, then the following three assertions hold:
\begin{enumerate}[{\rm (i)}]
	\item\label{NecessaryConditions:CentreField} $Z(\Skew)$ is a field;
	\item\label{NecessaryConditions:GSimple} $A$ is $G$-simple;
	\item\label{NecessaryConditions:SigmaInjective} $\sigma : G \to \Aut(A)$ is injective.
\end{enumerate}
\end{proposition}

\begin{proof}
\eqref{NecessaryConditions:CentreField}:
Let $a\in Z(\Skew)\setminus \{0\}$. By the simplicity of $R$ we get $aR=Ra=R$, which shows that
$a$ is invertible. One easily checks that $a^{-1}$ belongs to $Z(\Skew)$.\\
\eqref{NecessaryConditions:GSimple}:
Let $J$ be a non-zero proper $G$-invariant ideal of $A$.
Then $J\rtimes_\sigma G$ is a non-zero ideal of $\Skew$.
By simplicity of $\Skew$ we get $J\rtimes_\sigma G = \Skew$
and hence $A\subseteq J \rtimes_\sigma G$. Thus $A \subseteq J$. This shows that $J=A$
and hence $A$ is $G$-simple.\\
\eqref{NecessaryConditions:SigmaInjective}:
Let $g\in \ker(\sigma)$ be arbitrary and denote by $I$ the two-sided ideal of $\Skew$
generated by the element $u_e - u_g$.
Note that for any $s,t\in G$ and
$a_s,b_t\in A$ we get
\begin{equation}\label{zeroaugmentation}
a_s u_s (u_e-u_g)b_t u_t = a_s u_s b_t (u_e-u_g)u_t
 = a_s \sigma_s(b_t) (u_{st}-u_{sgt}).
\end{equation}
Clearly $\epsilon(I)=\{0\}$.
Since $\epsilon\lvert_A : A \to A$ is injective we conclude that $I\cap
A = \{0\}$, which shows that $I\subsetneq \Skew$.
By the simplicity of $\Skew$ we conclude that $I=\{0\}$.
In particular $u_e - u_g =0$ and hence $g=e$.
This shows that $\ker(\sigma)=\{e\}$.
\end{proof}

\begin{remark}
Assertions (i)-(iii) above are in general not sufficient to guarantee simplicity of $\Skew$ (see Example \ref{MotExempel}).
\end{remark}

\section{The ideal intersection property for $Z(\Skew)$}

The following lemma, which was inspired by \cite{JonasJohan}, plays a
key role in the present article.

\begin{lemma}\label{SwitchIdealElement}
Let $R=\Skew$ be a skew group ring where $G$ is abelian and $A$ is $G$-simple.
For each non-zero $r\in \Skew$ there exists some $r' \in \Skew$ with the following properties:
\begin{enumerate}[{\rm (i)}]
	\item $r'\in RrR$;
	\item $E(r')=1$;
	\item $|\Supp(r')|\leq |\Supp(r)|$.
\end{enumerate}
\end{lemma}

\begin{proof}
Take an arbitrary non-zero element $r\in R$. We can write $r=\sum_{g\in G}a_g u_g$, where $a_g\in A$ is zero for all but finitely many $g\in G$.
Since $r$ is non-zero we can choose some $h\in G$ such that $a_h\neq 0$. The element $r u_{h^{-1}}$ clearly belongs to $RrR$ and we note that
$|\Supp(ru_{h^{-1}})|=|\Supp(r)|$ and $E(r u_{h^{-1}})=a_h \neq 0$. Thus, without loss of generality, we may replace $r$ by $r u_{h^{-1}}$
and can therefore assume that $r=\sum_{g\in G}a_g u_g$ is such that $a_e\neq 0$.
The set
\[J=\big\{E(s) \mid s \in RrR \text{ such that } \Supp(s)\subseteq \Supp(r) \big\}\]
contains the non-zero element $a_e$ (since $r\in RrR$) and hence $J$ is a non-zero ideal of $A$.
We claim that $J$ is $G$-invariant. Indeed, if $a\in J$, then $a+\sum_{g\in\Supp(r)\setminus\{e\}}b_g u_g\in RrR$ for some $b_g\in A$, $g\in \Supp(r)\setminus\{e\}$. For any $h\in G$, we get
\[RrR\ni u_h (a+\sum_{g\in\Supp(r)\setminus\{e\}} b_g u_g) u_{h^{-1}} = \sigma_h(a) + \sum_{g\in\Supp(r)\setminus\{e\}} \sigma_h(b_g) \underbrace{u_{hgh^{-1}}}_{=u_g} \]
which yields $\sigma_h(a)\in J$. This shows that $J$ is $G$-invariant. By the $G$-simplicity of $A$ we conclude that $1\in A=J$. Hence there is some $r':=1+\sum_{g\in\Supp(r)\setminus\{e\}}b_g u_g\in RrR$, for some $b_g \in A $, $g\in \Supp(r) \setminus \{e\}$, which clearly satisfies (i)-(iii).
\end{proof}

Recall from \cite{OinertLundstrom} that a subring $S$ of a ring $T$ is said to have the \emph{ideal intersection property} (in $T$)
if $S\cap I \neq \{0\}$ holds for each non-zero ideal $I$ of $T$.

\begin{proposition}\label{CenterIntersection}
Let $R=\Skew$ be a skew group ring where $G$ is abelian and $A$ is $G$-simple.
Every non-zero ideal of $R$ has non-empty intersection with $Z(R)\cap \big(1+\sum_{g\in G\setminus\{e\}} A u_g\big)$. In particular, $Z(R)$
has the ideal intersection property in $R$.
\end{proposition}
\begin{proof}
Let $I$ be a non-zero ideal of $R$. Choose some $r\in I \setminus \{0\}$ such that $|\Supp(r)|$ is as small as possible. By Lemma \ref{SwitchIdealElement} there exists some $r' \in RrR\subseteq I$ such that $E(r')=1$ and $|\Supp(r')| \leq |\Supp(r)|$. In fact, by minimality of $|\Supp(r)|$ among all non-zero elements of $I$, we have $|\Supp(r')|=|\Supp(r)|$. Let $a\in A$ be arbitrary. Note that $E(r'a-ar')=a-a=0$ and thus $|\Supp(r'a-ar')|<|\Supp(r')|$. By the minimality of $|\Supp(r')|$ and the fact that $r'a-ar'\in I$ we conclude that $r'a-ar'=0$. This shows that $r'$ belongs to the centralizer of $A$. Now, let $g\in G$ be arbitrary. Note that $E(u_g r' u_g^{-1} - r')=1-1=0$ and thus $|\Supp(u_g r' u_g^{-1} - r')|<|\Supp(r')|$. Again, since $u_g r' u_g^{-1}-r'\in I$, by the minimality of $|\Supp(r')|$ we get $u_g r' u_g^{-1}-r'=0$. This shows that $u_g r'=r' u_g$, for all $g \in G$. Since $R=\Skew$ is generated as a ring by the elements of $A$ and $\{u_g\}_{g\in G}$, we conclude that $r'\in I\cap Z(R)\cap \big(1+\sum_{g\in G\setminus\{e\}} A u_g\big)$.
\end{proof}

\noindent The following lemma can sometimes be used to decide if $Z(\Skew)$ is a field or not.

\begin{lemma}\label{CentrumEkvivalenser}
Let $\Skew$ be a skew group ring. Consider the following assertions:
\begin{enumerate}[{\rm (i)}]
	\item $Z(\Skew)\subseteq A$;
	\item $Z(\Skew)=A^G \cap Z(A)$;
	\item[{\rm (iii)}] $Z(\Skew)$ is a field.
\end{enumerate}
The following conclusions hold:
\begin{enumerate}[{\rm (a)}]
	\item\label{CentrumEkvivalenser:A} {\rm (i)} and {\rm (ii)} are equivalent;
	\item\label{CentrumEkvivalenser:B} if $A$ is $G$-simple, then {\rm (i)} and {\rm (ii)} imply {\rm (iii)};
	\item\label{CentrumEkvivalenser:C} if $G$ is an orderable abelian group, then {\rm (iii)} implies {\rm (i)} and {\rm (ii)}.
\end{enumerate}
\end{lemma}

\begin{proof}
(a)
(i)$\Rightarrow$(ii): Let $a\in Z(\Skew)\subseteq A$. Then $a u_g = u_g a$ holds for all $g\in G$. Hence $(a-\sigma_g(a)) u_g=0$, or equivalently
$a=\sigma_g(a)$, for all $g\in G$. Hence $Z(\Skew)\subseteq A^G \cap Z(A)$.
The other inclusion is straightforward.\\
(ii)$\Rightarrow$(i): This is trivial.\\
(b)
(ii)$\Rightarrow$(iii): Suppose that $A$ is $G$-simple. Let $a\in A^G \cap Z(A)$ be non-zero. Then $Aa$ is a non-zero $G$-invariant ideal of $A$. Thus $Aa=A$. In particular, $1 \in Aa$, which shows that $a$ is invertible in $A$ and one can easily check that the inverse of $a$ belongs to $A^G \cap Z(A)$.\\
(c)
(iii)$\Rightarrow$(i): Suppose that $G$ is an orderable abelian group.
Assume that $Z(\Skew) \cap Au_g \neq \{0\}$ for some $g\in G \setminus \{e\}$ and take some non-zero $cu_g \in Z(\Skew) \cap Au_g$.
Then $1+cu_g \in Z(\Skew) \setminus \{0\}$ is invertible. Using that $G$ is an orderable group, we may without loss of generality assume that $g>e$. Let $r$ be the inverse of $1+cu_g$ and write $r=r_{h_1}u_{h_1}+\ldots+r_{h_k}u_{h_k}$, where $r_{h_i} \in A\setminus\{0\}$ for some distinct $h_1,\ldots,h_k\in G$ such that $h_1 < \ldots < h_k$. The term of lowest degree in the product $(1+cu_g)r$ is $1r_{h_1}u_{h_1}$, and the term of highest degree is $cu_gr_{h_k}u_{h_k}=c\sigma_{g}(r_{h_k})u_{gh_k}$, which is non-zero since $cu_g$ is invertible.
On the other hand, $(1+cu_g)r=1$ is homogeneous and therefore $k=1$. Hence $r_{h_1}u_{h_1}+c\sigma_{g}(r_{h_k})u_{gh_k}=1$, but this is a contradicton since $g>e$. Hence $Z(\Skew)\subseteq A$.
\end{proof}

\begin{example}[Inner actions and simplicity]\label{InnerExample}
Let $A=M_2(\R)$ and $G=\Z/2\Z$.
Put $M=\bigl( \begin{smallmatrix} 0 & 1 \\ -1 & 0 \end{smallmatrix} \bigr)$
and note that $M^2=-I$.
Define $\sigma : G \to \Aut(A)$ by $\sigma_0=\identity_A$ and $\sigma_1(a)=MaM^{-1}$ for $a\in A$.
The action of $G$ is clearly \emph{inner}.
We claim that $\Skew$ is a simple ring. Since $A$ is simple, and therefore $G$-simple,
it follows by Theorem \ref{MainTheorem}\eqref{MainTheorem:C} that it is enough to show that $Z(\Skew)$ is a field.
By a straightforward calculation we get
\begin{displaymath}
	Z(\Skew)=\{ a_0 I u_0 + a_1 M u_1 \in \Skew \mid a_0,a_1 \in \R \}
\end{displaymath}
which is a field. Indeed, let $a_0 I u_0 + a_1 M u_1$ be an arbitrary non-zero element of $Z(\Skew)$.
Using the fact that $a_0^2+a_1^2 \neq 0$
it follows by elementary linear algebra that the equation
$(a_0 I u_0 + a_1 M u_1)(b_0 I u_0 + b_1 M u_1)=1$
always has a unique solution $(b_0,b_1) \in \R^2$.
\end{example}

\begin{remark}
In \cite[Proposition 2.1]{Crow05} Crow proved the claim of Corollary \ref{CrowCorollary} below.
Example \ref{InnerExample} shows that outerness of the action is not a necessary condition for simplicity of the corresponding skew group ring.
It also motivates the need for Theorem \ref{MainTheorem}\eqref{MainTheorem:C}.
\end{remark}

We shall now show that \cite[Proposition 2.1]{Crow05} can easily be obtained as a corollary of Theorem \ref{MainTheorem}.

\begin{corollary}\label{CrowCorollary}
Let $\Skew$ be a skew group ring where $G$ is abelian and outer.
The following two assertions are equivalent:
\begin{enumerate}[{\rm (i)}]
	\item\label{CrowCorollary:Simple} $\Skew$ is a simple ring;
	\item\label{CrowCorollary:GSimple} $A$ is $G$-simple.
\end{enumerate}
\end{corollary}

\begin{proof}
\eqref{CrowCorollary:Simple}$\Rightarrow$\eqref{CrowCorollary:GSimple}: This follows from Theorem \ref{MainTheorem}\eqref{MainTheorem:A}.\\
\eqref{CrowCorollary:GSimple}$\Rightarrow$\eqref{CrowCorollary:Simple}:
By Lemma \ref{CentrumEkvivalenser}\eqref{CentrumEkvivalenser:B} and
Theorem \ref{MainTheorem}\eqref{MainTheorem:C} it is enough to show that $Z(\Skew)\subseteq A$.
Let $r=\sum_{g\in G} a_g u_g$ be an arbitrary non-zero element of $Z(\Skew)$.
Take $g\in G$ such that $a_g\neq 0$.
Since $r\in Z(\Skew)$ we conclude that $a_g \in A^G$ and
\begin{equation}\label{ACommutationRelation}
	ba_g = a_g \sigma_g(b), \quad \forall b\in A.
\end{equation}
Using that $a_g\in A^G$ it is clear that the set
$J:=A a_g A = A a_g = a_g A$ is a non-zero $G$-invariant ideal of $A$
and hence $J=A$. Thus, $1=a_g c$ for some $c\in A$.
From \eqref{ACommutationRelation} we get
$1=\sigma_g(1)=a_g \sigma_g(c) = c a_g$, which shows that $a_g$ is invertible.
Therefore $\sigma_g(b)=a_g^{-1} b a_g$ for all $b\in A$, so $\sigma_g$ is inner.
We now conclude that $g=e$.
\end{proof}

\section{Injectivity of $\sigma : G \to \Aut(A)$ and maximal commutativity of $A$}

Maximal commutativity of $A$ in $\Skew$ implies injectivity of $\sigma : G \to \Aut(A)$.
If $A$ is e.g. an integral domain, then it is easy to see that the two assertions are equivalent.
The same conclusion does not, however, hold for an arbitrary commutative ring $A$.
The following proposition describes a situation in which the two assertions are in fact equivalent.

Let $K$ denote the kernel of the group homomorphism $\sigma : G \to \Aut(A)$.

\begin{proposition}\label{MaxCommInjectivity}
Let $R=\Skew$ be a skew group ring where $G$ is an abelian group and $A$ is a commutative and $G$-simple ring.
Then $C_R(A)=A \rtimes_{\sigma} K$.
In particular, $A$ is a maximal commutative subring of $\Skew$
if and only if $\sigma$ is injective.
\end{proposition}

\begin{proof}
Let $\sum_{g\in K} a_g u_g$ be an arbitrary element of $A \rtimes_{\sigma} K$.
For any $a \in A$ we have
$a \sum_{g\in K} a_g u_g = \sum_{g\in K} a_g \sigma_g(a) u_g = \sum_{g\in K} a_g u_g a$.
This shows that $A \rtimes_{\sigma} K \subseteq C_R(A)$.
Now let $\sum_{g\in G} a_g u_g \in C_R(A) \setminus \{0\}$ be arbitrary.
Take $h\in G$ such that $a_h\neq 0$. Note that $a_h u_h \in C_R(A)$ since $C_R(A)$ is $G$-graded.
Consider the set
\begin{displaymath}
	I=\{b\in A \mid (\sigma_h(a)-a)b = 0, \,\, \forall a\in A\}.
\end{displaymath}
It is clear that $I$ is an ideal of $A$ and it is non-zero since $a_h\in I$.
Take $b\in I$ and $g\in G$.
Then $(\sigma_g(\sigma_h(a))-\sigma_g(a))\sigma_g(b)=0$ or equivalently
$(\sigma_h(\sigma_g(a))-\sigma_g(a))\sigma_g(b)=0$ holds for all $a\in A$.
This shows that $(\sigma_h(c)-c)\sigma_g(b)=0$ holds for all $c\in A$
and hence $\sigma_g(b)\in I$. Thus, $I$ is $G$-invariant.
By assumption we get $I=A$, so $1\in I$ which yields $\sigma_h=\identity_A$, i.e. $h\in K$.
Since $h$ was arbitrarily chosen we conclude that $C_R(A)\subseteq A \rtimes_{\sigma} K$.
\end{proof}

\section{An application to topological dynamical systems}\label{DynamicalSection}

Let $(X,G)$ be a \emph{transformation group} consisting of a topological group $G$ acting on a compact Hausdorff space $X$.
Furthermore, assume that the action $G\curvearrowright X$ is \emph{strongly continuous}, i.e.
the map $G \times X \to X$, $(g,x)\mapsto g.x$ is continuous with respect to the respective topologies.

The algebra of complex-valued continuous functions
on $X$, where the operations of addition and multiplication are defined pointwise, is denoted by $C(X)$.
We define $\lvert\lvert f \lvert\lvert := \sup_{x\in X} |f(x)|$, for $f\in C(X)$, and one easily checks that
this defines a norm on $C(X)$ which turns it into a unital
$C^*$-algebra.

\noindent The transformation group $(X,G)$ induces a group homomorphism
\begin{equation}\label{SigmaInduced}
	\sigma : G \to \Aut(C(X)), \quad
        \sigma_g(f)(x)=f(g^{-1}.x), \quad g\in G, \,\, f\in C(X), \,\, x\in X.
\end{equation}
It follows by the strong continuity of the action, that the automorphisms $\sigma_g\in \Aut(C(X))$, for $g\in G$, are all continuous.
We call $C(X)\rtimes_{\sigma} G$ the \emph{skew group
  algebra\footnote{The completion of this skew group algebra with respect to a suitable norm would be called a {\it crossed product $C^*$-algebra} by $C^*$-algebraists. In non-commutative ring theory, however, a skew group algebra is a special case of the more general (algebraic) \emph{\emph{crossed product}} construction.} associated to the transformation group $(X,G)$.}

\begin{definition}
If there, for each $g\in G\setminus \{e\}$, exists some $x\in X$
such that $g.x \neq x$,
then the transformation group $(X,G)$ is said to be \emph{faithful}.
A subset $V\subseteq X$ is said to be \emph{$G$-invariant} if $g.V
\subseteq V$ holds for all $g\in G$.
If $X$ contains no non-empty proper closed $G$-invariant
subset,
then the
transformation group $(X,G)$ is said to be \emph{minimal}.
\end{definition}

\begin{remark}
Note that a subset $V\subseteq X$ is $G$-invariant if and only if $g.V=V$ holds for all $g\in G$.
Minimality of $(X,G)$ may equivalently be stated as saying that for each $x\in X$ the
orbit of $x$, i.e. the set $\{g.x \mid g\in G\}$, is a dense subset of $X$.
\end{remark}

Let $\mathcal{P}_{\cl}(X)$ denote the set of all closed subsets of $X$, and $\Ideal_{\cl}(C(X))$ denote the set of all closed ideals of $C(X)$. There is a one-to-one correspondence between these sets. Indeed, consider the map
\begin{displaymath}
	\Ideal_{\cl}(C(X)) \ni I \stackrel{\varphi}{\mapsto} \{x\in X \mid f(x)=0 \text{ for all } f\in I\} \in \mathcal{P}_{\cl}(X)
\end{displaymath}
and the map
\begin{displaymath}
	\mathcal{P}_{\cl}(X) \ni V \stackrel{\psi}{\mapsto} \{f\in C(X) \mid f\lvert_V\equiv0\} \in \Ideal_{\cl}(C(X)).
\end{displaymath}
It follows that $\varphi$ and $\psi$ are well-defined and
that $\psi\circ\varphi=\identity_{\Ideal_{\cl}(C(X))}$ and $\varphi\circ\psi=\identity_{\mathcal{P}_{\cl}(X)}$.

\begin{lemma}\label{FaithfulnessInjectivity}
$(X,G)$ is faithful if and only if $\sigma$ (defined by \eqref{SigmaInduced}) is injective.
\end{lemma}

\begin{proof}
Note that if $|X|=1$, then both assertions are equivalent.
Let us therefore assume that $|X|>1$.
If $(X,G)$ is not faithful, then there is some $g\in G\setminus \{e\}$ such that $g.x=x$ for all $x\in X$.
It then follows by \eqref{SigmaInduced} that $\sigma_{g^{-1}} = \identity_{C(X)}$, thus $\sigma$ is not injective.
Conversely, let $(X,G)$ be faithful. Seeking a contradiction, suppose that $\sigma$ is not injective.
There is some $g\in G \setminus \{e\}$ such that $f(g^{-1}.x)=f(x)$ for all $f\in C(X)$ and $x\in X$.
Since $(X,G)$ is faithful, there is some $x\in X$ such that $g^{-1}.x\neq x$.
By Urysohn's lemma (and the fact that $|X|>1$) we conclude that there is some $f : X \to [0,1] \subseteq \C$ such that
$f(g^{-1}.x)\neq f(x)$. This is a contradiction.
\end{proof}

\begin{lemma}\label{MinimalEquivalences}
The following four assertions are equivalent:
\begin{enumerate}[{\rm (i)}]
	\item $(X,G)$ is minimal;
	\item There is no non-empty closed proper $G$-invariant subset of $X$;
	\item $C(X)$ is $G$-simple with respect to closed ideals;
	\item $C(X)$ is $G$-simple.
\end{enumerate}
\end{lemma}

\begin{proof}
(i)$\Leftrightarrow$(ii):
This is indeed the definition.\\
(ii)$\Leftrightarrow$(iii):
Note that $\varphi$ and $\psi$ also give rise to a one-to-one correspondence between closed $G$-invariant subsets of $X$
and closed $G$-invariant (with respect to $\sigma$) ideals of $C(X)$.\\
(iii)$\Rightarrow$(iv):
Suppose that $C(X)$ is $G$-simple with respect to closed ideals. Let
$I$ be a non-zero $G$-invariant ideal of $C(X)$. We wish to show that $I=C(X)$.
Denote by $\bar{I}$ the closure of $I$, and note that this is also an ideal of $C(X)$.
The maps $\sigma_g : C(X) \to C(X)$, for $g\in G$, are continuous
and hence the $G$-invariance of $I$ implies $\sigma_g(\bar{I}) \subseteq \bar{I}$, for $g\in G$.
This shows that $\bar{I}$ is a $G$-invariant (and closed) ideal of $C(X)$. By the assumption we get $\bar{I}=C(X)$.
Since $C(X)$ is a unital $C^*$-algebra (and in particular a Banach algebra),
the closure of any proper ideal is still a proper ideal.
Therefore we conclude that $I=C(X)$.\\
(iv)$\Rightarrow$(iii):
This is trivial.
\end{proof}

\subsection{A faithful, minimal and non-free action of an ICC group}\label{Section:MotExempel}

Recall that a group $G$ is said to be an \emph{ICC group} if it has the \emph{infinite conjugacy class property},
i.e. for each $g\in G \setminus \{e\}$ the set $\{hgh^{-1} \mid h\in G\}$ is infinite.
Clearly, finite groups and abelian groups can not be ICC.

\begin{proposition}\label{CenterContainedInAForICC}
Let $\Skew$ be a skew group ring. If $G$ is an ICC group, then $Z(\Skew)\subseteq A$.
\end{proposition}

\begin{proof}
Let $r = \sum_{g\in G} a_g u_g$ be an element of $Z(\Skew)$.
For any $h\in G$ we have
\begin{displaymath}
\sum_{g\in G} a_g u_g = u_h \left( \sum_{g\in G} a_g u_g \right) u_h^{-1} = \sum_{g\in G} \sigma_h(a_g) u_{hgh^{-1}}
= \sum_{s\in G} \sigma_h(a_{h^{-1}sh}) u_s.	
\end{displaymath}
Take $g\in \Supp(r)$ and note that $a_g = \sigma_h(a_{h^{-1}gh})\neq 0$ for all $h\in G$.
Since $G$ is an ICC group and $\Supp(r)$ is finite we get $g=e$.
This shows that $Z(\Skew)\subseteq A$.
\end{proof}

Given a transformation group $(X,G)$ and $x\in X$
we let $\Stab_G(x):=\{g\in G \mid g.x=x\}$ denote the \emph{stabilizer subgroup of $x$} in $G$.

\begin{lemma}\label{ICClemma}
Let $G$ be a group which acts faithfully on a set $X$.
If the set $\Stab_G(x).y$ is infinite for any
two $x,y\in X$ such that $x\neq y$,
then $G$ is an ICC group.
\end{lemma}

\begin{proof}
Let $g\in G \setminus \{e\}$.
Then there is some $x\in X$ such that $y:= g.x \neq x$.
For any $h\in \Stab_G(x)$ we have
$hgh^{-1}.x=h.(g.x)=h.y$.
By the assumption $\{hgh^{-1}.x \mid h\in \Stab_G(x)\}$
is infinite, so in particular $G$ is an ICC group.
\end{proof}

\begin{proposition}\label{HomeoIsICC}
$\Homeo(S^1)$, the group of all homeomorphisms of the circle $S^1$, is an ICC group.
\end{proposition}

\begin{proof}
The group $G=\Homeo(S^1)$ acts on $X=S^1$ in an obvious way and this action is clearly faithful.
Let $x,y\in S^1$ such that $x\neq y$. Take any $z\in S^1$ such that $z\neq x$.
We now define an invertible piecewise linear map $f_z : S^1 \to S^1$
satisfying $f_z(y)=z$ and $f_z \in \Stab_G(x)$. This is always possible since $z\neq x$.
We can choose $z$ in infinitely many ways and hence $\Stab_G(x).y$ is infinite.
By Lemma \ref{ICClemma}, the desired conclusion follows.
\end{proof}

\begin{example}\label{MotExempel}
Let $X=S^1$ be the circle, $G=\Homeo(S^1)$ the group of all homeomorphisms of $S^1$
and consider the skew group algebra $C(X)\rtimes_{\sigma} G$ where $\sigma$ is defined by \eqref{SigmaInduced}.
It is easy to see that the action $G\curvearrowright X$ is faithful and minimal.
Hence, by Lemma \ref{FaithfulnessInjectivity} $\sigma$ is injective
and by Lemma \ref{MinimalEquivalences} $C(X)$ is $G$-simple.

By Proposition \ref{HomeoIsICC}, $G$ is an ICC group and
by combining Proposition \ref{CenterContainedInAForICC} and Lemma \ref{CentrumEkvivalenser}\eqref{CentrumEkvivalenser:B},
we conclude that $Z(C(X)\rtimes_{\sigma} G)$ is a field.

We claim that $C(X)\rtimes_{\sigma} G$ is not simple. By Theorem \ref{SimplicityMaxComm}
we need to show that $C(X)$ is not a maximal commutative subalgebra of $C(X)\rtimes_{\sigma} G$.
To see this, take $g\in G\setminus \{e\}$ such that $g^{-1}(x)=x$ for all $x\in [0,\frac{1}{2}]$. Choose a non-zero $f_g\in C(X)$
such that $f_g(x)=0$ for all $x\in [\frac{1}{2},1]$. Then it follows that $(\sigma_g(b)-b)f_g=0$
for any $b\in C(X)$. Hence $f_g u_g$ commutes with each element of $C(X)$.
This shows that $C(X)$ is not a maximal commutative subalgebra.
\end{example}

\begin{remark}
A minimal and faithful action of an abelian group on a compact Hausdorff space
is necessarily \emph{free}, in the sense that if $g\in G \setminus \{e\}$ then for any $x\in X$ we have $g.x\neq x$.
The action in Example \ref{MotExempel} is clearly non-free.
\end{remark}

\section{Proof of the main results}

We are now fully prepared to 
prove the main results of this article.

\begin{proof}[Proof of Theorem \ref{MainTheorem}]
\eqref{MainTheorem:A} This follows immediately from Proposition \ref{NecessaryConditions}. \\
\eqref{MainTheorem:B} Consider Example \ref{MotExempel}.
In the example it is explained that \eqref{MainTheorem:FieldGSimple} and \eqref{MainTheorem:InjectiveGSimple} hold,
but that \eqref{MainTheorem:Simple} fails to hold. \\
\eqref{MainTheorem:C} We need to show that \eqref{MainTheorem:FieldGSimple} implies \eqref{MainTheorem:Simple}.
The rest follows from \eqref{MainTheorem:A}.
Suppose that $A$ is $G$-simple and that $Z(\Skew)$ is a field.
Let $I$ be a non-zero ideal of $\Skew$. By Proposition \ref{CenterIntersection} we conclude that
$I\cap Z(\Skew) \neq \{0\}$. Hence $1\in I$ and therefore $I=\Skew$. This shows that $\Skew$ is simple.\\
\eqref{MainTheorem:D} We need to show that \eqref{MainTheorem:InjectiveGSimple} implies \eqref{MainTheorem:Simple}.
The rest follows from \eqref{MainTheorem:A} and \eqref{MainTheorem:C}. \\
Suppose that $A$ is $G$-simple and that $\sigma$ is injective.
By Proposition \ref{MaxCommInjectivity} $A$ is a maximal commutative subring of $\Skew$
and hence, by Theorem \ref{SimplicityMaxComm}, $\Skew$ is simple.
\end{proof}

\begin{proof}[Proof of Theorem \ref{TransGroupTheorem}]
\eqref{TransGroupTheorem:A} It follows from Theorem \ref{SimplicityMaxComm}
that \eqref{TransGroupTheorem:Simple} and \eqref{TransGroupTheorem:GSimpleMaxComm} are equivalent.
The other claim follows from Theorem \ref{MainTheorem}\eqref{MainTheorem:A}, Lemma \ref{FaithfulnessInjectivity} and Lemma \ref{MinimalEquivalences}.\\
\eqref{TransGroupTheorem:B} This follows immediately from Lemma \ref{FaithfulnessInjectivity} and Lemma \ref{MinimalEquivalences}.\\
\eqref{TransGroupTheorem:C} Consider Example \ref{MotExempel}.
In the example it is explained that
\eqref{TransGroupTheorem:GSimpleField}, \eqref{TransGroupTheorem:GSimpleInjective} and
\eqref{TransGroupTheorem:MinimalFaithful} hold,
but that \eqref{TransGroupTheorem:Simple} and \eqref{TransGroupTheorem:GSimpleMaxComm} fail to hold.\\
\eqref{TransGroupTheorem:D} This follows from \eqref{TransGroupTheorem:A}, \eqref{TransGroupTheorem:B} and Theorem \ref{MainTheorem}\eqref{MainTheorem:D}.
\end{proof}

\begin{remark}
If $A$ is commutative, then injectivity of $\sigma$ clearly implies outerness of $G$.
Hence, an alternative way of proving that \eqref{MainTheorem:InjectiveGSimple} implies \eqref{MainTheorem:Simple} in Theorem \ref{MainTheorem} (under the assumption that $A$ is commutative and $G$ is abelian),
is by applying Corollary \ref{CrowCorollary}.
\end{remark}

\section*{Acknowledgements}
The author is immensely grateful to Steven Deprez for stimulating discussions
and in particular for sharing his knowledge on ICC groups,
which gave rise to Section \ref{Section:MotExempel}.
The author is also grateful to Patrik Lundstr\"{o}m for stimulating discussions on the topic of this article.
This research was supported by The Swedish Research Council (postdoctoral fellowship no. 2010-918)
and
The Danish National Research Foundation (DNRF) through the Centre for Symmetry and Deformation.

\bibliographystyle{amsalpha}

\end{document}